\documentclass[b5paper,reqno]{amsart}
\usepackage{graphicx}
\usepackage{amsmath,amssymb,amsfonts,amsthm}
\usepackage[hidelinks]{hyperref}
\usepackage{mathrsfs,enumitem}
\usepackage[english]{babel}
\newtheorem{theorem}{Theorem}[section]

\newtheorem{proposition}[theorem]{Proposition}
\theoremstyle{definition}
\newtheorem{definition}[theorem]{Definition}

\newtheorem{counter example}[theorem]{Counter Example}

\newtheorem{corollary}[theorem]{Corollary}
\newtheorem{example}[theorem]{Example}
\numberwithin{equation}{section}
\DeclareRobustCommand{\rchi}{{\mathpalette\irchi\relax}}
\newcommand{\irchi}[2]{\raisebox{\depth}{$#1\chi$}}
\begin{document}
\Large{
		\title{ON A THEOREM OF GRZEGOREK AND LABUDA}
		
		\author[S. Basu]{Sanjib Basu}
		\address{\large{Department of Mathematics,Bethune College,181 Bidhan Sarani}}
		\email{\large{sanjibbasu08@gmail.com}}

	\author[N. Tamang]{Navdeep Tamang}
	\address{\large{Department of Pure Mathematics, University of Calcutta, 35, Ballygunge Circular Road, Kolkata 700019, West Bengal, India}}
	\email{\large{navdeeptamang007@gmail.com}}}

	\begin{abstract}
  This paper presents a generalized version of a theorem of Grzegorek and Labuda in category bases and also endeavours to establish a variant formulation of the same in Marczewski structures. 
	\end{abstract}
\subjclass[2020]{AMS}
\thanks{}
	\maketitle

\section{INTRODUCTION}

	In [3] Grzegorek and Labuda gave a common abstraction of two theorems of Sierpi\'{n}ski. These theorems which in turn are connected with the results and techniques  of Luzin and Novikov [5] are as follows:
    \begin{theorem}[\cite{11},Theorem II]
        Let Q be an infinite subset of $\mathbb{R}$. Then Q contains an infinte number of disjoint subsets, each of which has the same outer Lebesgue measure as Q.
    \end{theorem}
\begin{theorem} [\cite{10}, Supplement]
 Let Q be a subset of an interval J. If Q is everywhere 2nd category in J, then it is the union of an infinte number of disjoint subsets, each of which is everywhere 2nd category in J.
 
 (Here, Q is everywhere 2nd category in J means Q is 2nd category at each point x $\in$ J. In otherwords, for each neighbourhood V of x, the intersection V$\cap$Q is a set of 2nd category.)
\end{theorem}

To give a common abstraction of the above two theorems, Grzegorek and Labuda chose a triple($X,\mathscr{A}$,$\mathscr{I}$) as their basic framework, where $X$ is a nonempty set, $\mathscr{A}$ is a $\sigma$-algebra of its subsets (which are called measurable) and $\mathscr{I}$ is a $\sigma$-ideal of its subsets (which are called negligible) such that $\mathscr{I} \subset \mathscr{A}$. In this frameworks, given two subsets $E$ and $Q$ of $X$ with $E\subset$Q, they defined the notion ``$E$ is a full subset of $Q$" with respect to ($\mathscr{A}$,$\mathscr{I}$) which is a reformation of an earlier notion  of `same size'. used in [2]. Following [14] Grzegorek and Labuda also defined ($\mathscr{A}$,$\mathscr{I}$)- complete nonmeasurability (in $Q$) of a subset $E$ of $Q$ and proved (Proposition 1.4,[3]) that for $ E \subset Q$, $E$ is ($\mathscr{A}$,$\mathscr{I}$)- completely non-measurable in $Q$ iff $E$ and $Q\backslash$E are full subsets of $Q.$

We end this section by stating the two assumptions on ($X,\mathscr{A}$,$\mathscr{I}$) by Grzegorek and Labuda and also their main theorem  as stated in [3] which motivated us to work on this article.
\begin{itemize}
    \item The family $\mathscr{A}$$\backslash$$\mathscr{I}$ satisfies the countable chain condition(ccc), that is, every disjoint family of non-negligible measurable sets is countable. It is well-known and easily shown that under ccc each subsets of $X$ admits an envelope.

    \item $X$ has the non-separation property with respect to $\mathscr{A}$. Namely, if $E\notin \mathscr{I}$, then there exist disjoint subsets of $E$ which cannot be separated by sets from $\mathscr{A}$ (in particular, the singletons must belong to $\mathscr{I}$)
\end{itemize}
\begin{theorem}
Each infinite subset $E$ of $X$ decomposes into two disjoint full subsets. Consequently, it can be written as the infinte disjoint union of its full (equivalently-completely non- measurable) subsets.
\end{theorem}
The present paper extends Theorem 1.3 in the more general settings of category bases along with a variant formulation of the same in Marczewski structure. Here we find that the axiomatic foundation of category base based on the notion of region requires no assumptions of the above type.
\section{PRELIMINARIES AND RESULTS}
Category base is an abstract theory of point sets which unifies within a common framework substantial number of results concerning Lebesgue measure, Baire property and also some other aspects of point set classification. It was introduced by J.C Morgan II [8](see also [6]) and has been developed since then through a series of papers [7],[9],[12],[13] etc.
\begin{definition}[\cite{6}] A pair ($X,\mathscr{C}$)  where $X$ is a nonempty set and $\mathscr{C}$ is a family of subsets of $X$ is called a category base if the nonempty members of $\mathscr{C}$ called regions satisfy the following axioms:

\begin{enumerate}
\item Every point of $X$  is contained in at least one region; i.e, $X =\cup\mathscr{C}$
\item Let $A$ be a region  and  $\mathscr{D}$ be a non-empty family of disjoint regions having cardinality less than the cardinality of $\mathscr{C}$
\begin{enumerate}[label=\roman*]

    \item If $A\cap(\cup\mathscr{D})$ contains a region, then there exists a region    $D\in\mathscr{C}$ such that $A \cap D$ contains a region.
    \item If $A\cap(\cup\mathscr{D})$ contains no region, then there exist a  region $B\subset A$ which is disjoint from every region in $\mathscr{D}$. 
\end{enumerate}
\end{enumerate}
\end{definition}
Several examples of category bases may be found in the well-known monograph of Morgan[6] and also in the references stated above. In connection with the present work, we like  to highlight the following few.

\begin{example} 
$X=\mathbb{R}^n$ and $\mathscr{C}$ is the usual topology on $X$.
\end{example}
\begin{example}$ X=\mathbb{R}^n$ and $\mathscr{C}$ is the family of compact subsets of $X$ of positive Lebesgue measure.
\end{example}
\begin{example}
    
 $X=\mathbb{R}^n$ and $\mathscr{C}$ is family of all perfect subsets of $X$. This forms a perfect base or the Marczewski base.
\end{example}
\begin{example}$X\neq\phi$ and
$\mathscr{C}=\mathscr{A}\backslash\mathscr{I}$ where $\mathscr{A}$ is a $\sigma$-algebra and $\mathscr{I}(\subset\mathscr{A})$ is a $\sigma$-ideal of subsets of $X$ such that $\mathscr{A}\backslash\mathscr{I}$ satisfies countably chain condition.
\end{example}
\begin{definition}
     In a category base (X,$\mathscr{C}$), a set is called 
\begin{enumerate}[label=(\roman*)]
    \item singular,  if every region contains a subregion disjoint from the set itself;
    \item meager,if it can be expressed as a countable union of singular sets. A category base in which every singleton set is meager is called a `point-meager base';
    \item abundant, if it is not meager. A category base in which every region is abundant is called a `Baire base';
    \item a Baire set, if every region contains a subregion in which either the set or its complement is meager.
\end{enumerate}
\end{definition}
 The class of all meager (resp, Baire) sets is  here denoted by the symbol $\mathscr{M}(\mathscr{C}) $(resp$\mathscr{B}(\mathscr{C})$).
 We further add that in a category base ($X,\mathscr{C})$,

 \begin{definition}[\cite{4}] A subfamily $\mathscr{C'}$ of $\mathscr{C}$ is called a $\pi$-base if every region $A$ $\in\mathscr{C}$ contains a subregion $B$ $\in\mathscr{C'}$.
\end{definition}
\begin{definition}[\cite{1}]
    
A set $F$ is  a full subset of $E$ if $F\subset E$ and for every abundant Baire set $B$, $B\cap F$ is abundant whenever $B\cap E$ is so. If $E$ is a Baire set, this is equivalent to the assertion that $E\backslash F$ cannot contain any abundant Baire set. In this situation, we call $E$ the `Hull of $F$' and it is unique in the sense that if $E_1$ and $E_2$ are any two hull of $F$, then $E_1\triangle E_2 \in \mathscr{M}(\mathscr{C})$.
\end{definition}
 It may be noted here that the triple($X,\mathscr{A}$,$\mathscr{I})$ used by Grzegorek and Labuda gives rise to a category base (Example 2.5) because according to their first assumption $\mathscr{C}$=$\mathscr{A}\backslash\mathscr{I}$ satisfies countable chain condition whereas, on the otherhand if ($X,\mathscr{C}$) is a category base where $\mathscr{C}$ satisfies countable chain condition, then it may be associated with it the triple ($X,\mathscr{B}(\mathscr{C}),\mathscr{M}(\mathscr{C})$) where $\mathscr{B}(\mathscr{C}) \backslash$$\mathscr{M}(\mathscr{C})$ satisfies countable chain condition(Theorem 5,III,Ch 1, [6]). Thus category base presents a more basic and wider platform on which to generalize the theorem of Grzegorek and Labuda.

To establish our main results, we make the following assumptions:
\begin{enumerate}[label=(\roman*)]
    \item card(X) is a regular cardinal.
    \item Our category base ($X,\mathscr{C}$) possesses a $\pi$-base $\mathscr{C'}$
    with card($\mathscr{C'}$) not exceeding card($X$).
    \item $\mathscr{M}(\mathscr{C})$ is card($X$)-additive which means 
    that $\cup\mathscr{E}\in\mathscr{M}(\mathscr{C})$ whenever $\mathscr{E}\subset\mathscr{M}(\mathscr{C})$ and card($\mathscr{E}$)$<$card($X$).
    \item card($X$)=min$\lbrace\rchi : \mathscr{C}$ is $\rchi$-saturated$\rbrace$ where by the phrase ``$\mathscr{C}$ is $\rchi$-saturated" we mean that if $\mathscr{E}\subset\mathscr{C}$ such that card($\mathscr{E})=\rchi$, then there are at least two distinct members $E$, $F(\in\mathscr{E}$) such that $E\cap F\neq\phi$.
    \item Every abundant set in ($X,\mathscr{C}$) possesses a hull.
\end{enumerate}

Let $\mathscr{K}$ denote the family of all sets whose complements are members of $\mathscr{C'}$. Let $\bigcap\limits_{<card(X)}\mathscr{K}$: all sets representable as intersection of subfamilies of $\mathscr{K}$ whose cardinality is less than card($X$).
\begin{theorem}
Any singular set in ($X,\mathscr{C})$ is contained in a singular set which belongs to the family $\bigcap\limits_{<card(X)}\mathscr{K}$. Any  set in $\mathscr{M}(\mathscr{C})$ is  a subset of some set in $\mathscr{M}(\mathscr{C})\bigcap(\bigcup\limits_{\sigma}\bigcap\limits_{<card(X)}\mathscr{K})$.
\end{theorem}
\begin{proof}
Let $A$ be a singular set in ($X,\mathscr{C})$. Since $\mathscr{C'}$
is a $\pi$-base, for every $C\in\mathscr{C}$, there is a set $D\in\mathscr{C'}$ such that $D\subset$C and $D$ is disjoint from $A$. This constitutes a subfamily $\mathscr{N}$ of $\mathscr{C}$ satisfying the above property such that every member of $\mathscr{C}$ contains a member of $\mathscr{N}$. Let $Y=\cup\mathscr{N}$. Then ($Y,\mathscr{N}$) is a category base and by lemma 4,II,Ch 1,[6] a subfamily $\mathscr{M}$ of $\mathscr{N}$  consisting of mutually disjoint sets can be selected such that for every $N\in\mathscr{N}$, there exists $M\in\mathscr{M}$ such that $N\cap$M contains a region. This makes $Y\backslash$($\cup\mathscr{M}$) singular in ($Y,\mathscr{N}$) and consequently by Lemmma 3,II,Ch 1, [6], $X\backslash(\cup\mathscr{M})$ is singular in ($X,\mathscr{C}$). Moreover, $A\subset X\backslash(\cup\mathscr{M})$.

 Now by our assumption (iv) $X\backslash(\cup\mathscr{M})$ belongs to the family
 $\bigcap\limits_{<card(X)}\mathscr{K}$ which shows that A is contained in a singular set which belongs to the family $\bigcap\limits_{<card(X)}\mathscr{K}$. Consequently, any member of $\mathscr{M}(\mathscr{C})$ is a subset of some set in  $\mathscr{M}(\mathscr{C})\bigcap(\bigcup\limits_{\sigma}\bigcap\limits_{<card(X)}\mathscr{K})$.
\end{proof}
 It may be noted on the basis of assumptions (i) and (ii), that
 
 card( $\mathscr{M}(\mathscr{C})\bigcap(\bigcup\limits_{\sigma}\bigcap\limits_{<card(X)}\mathscr{K})$) does not exceed card($X$).

\begin{theorem}
 Any abundant set in ($X,\mathscr{C}$) contains a set of the form $F\backslash G$ where $F\in\mathscr{C}$  and $G\in$ $\mathscr{M}(\mathscr{C})\bigcap(\bigcup\limits_{\sigma}\bigcap\limits_{<card(X)}\mathscr{K})$.
\end{theorem}
 \begin{proof} Let $S\in\mathscr{B(\mathscr{C})}\backslash\mathscr{M}(\mathscr{C})$. By Theorem 2, III, Ch 1, [6], there exists $C\in\mathscr{C}$ such that $C\backslash S\in\mathscr{M}(\mathscr{C})$. By Theorem 2.9 again, $C\backslash S\subset G\in$ $\mathscr{M}(\mathscr{C})\bigcap(\bigcup\limits_{\sigma}\bigcap\limits_{<card(X)}\mathscr{K})$. Let $F\in\mathscr{C'}$ such that $F\subset C$. The theorem is now proved by choosing the set $F\backslash G$.
\end{proof}
 \begin{theorem} Let E be an abundant set in a Baire base (X,$\mathscr{C}$) and $E=\bigcup\limits_{t\in T}E_{t}$ is a partition of E into meager sets in  (X,$\mathscr{C}$). Then there exist two disjoint subfamilies $\Phi_{1}$ and $\Phi_{2}$ of T such that T=$\Phi_{1}\cup\Phi_{2}$ and each of the union $\bigcup\limits_{t\in\Phi_{1}}E_{t}$ and $\bigcup\limits_{t\in\Phi_{2}}E_{t}$ is full subsets of E. Consequently, there exists  a countably infinite sequence $ \lbrace \Phi_{n}\overunderset{\infty}{n=1}{\rbrace} $ of mutually disjoint subfamilies of T such that $\bigcup\limits_{n=1}^{\infty}\Phi_n$=T and each of the union $\bigcup\limits_{t\in\Phi_n}E_{t} $ is a full subset of E.
\end{theorem}
 To prove the above theorem, we need the following well-known auxilliary proposition [4].

 \begin{proposition}
     
 Let  $ \lbrace X_{i}{\rbrace}_{i\in I} $ be a family of subsets of an infinite set $X$ such that
 \begin{enumerate}[label=(\alph*)]
 \item card$ (X_{i})$=card(X) for each i$\in I$
 \item card(I)$\leq$card($X$).
\end{enumerate}

 Then there exist sets $Y, Z$ being a decomposition of $X$, such that card($Y\cap X_{i}$)=card($Z\cap X_{i}$)=card($X$) for each i$\in I$.
\end{proposition}
\begin{proof}
\ Let $S$ be the hull of $E$ which exist by assumption(v) . Then $S \in\mathscr{B}(\mathscr{C})\backslash\mathscr{M}(\mathscr{C})$ and for every $A\in \mathscr{B}(\mathscr{C})\backslash \mathscr{M}(\mathscr{C})$, $A\cap$E$\notin$ $\mathscr{M}(\mathscr{C})$ whenever $A\cap$S$\notin$ $\mathscr{M}(\mathscr{C})$. Let \mbox{$\mathscr{H}$=$\lbrace F\backslash G: F\in\mathscr{C'},G\in\mathscr{M}(\mathscr{C})\bigcap(\bigcup\limits_{\sigma}\bigcap\limits_{<card(X)}$
$\mathscr{K})$}
and $F\backslash G\subset S \rbrace$. By Theorem 2.10, every abundant Baire set in 
(X,$\mathscr{C})$ contained in $S$ contains a set from the family $\mathscr{H}$. For any $H\in\mathscr{H}$, let $Y(H)=\lbrace t\in T: E_t \cap H \neq \phi \rbrace$. Since ($X,\mathscr{C}$) is a Baire base, it follows from assumption (iii) and also the observation made at the end of the proof of Theorem 2.9 that card$(Y(H))=$card($X$) for every $H\in\mathscr{H}$ and moreover, card($\Sigma$)$\leq card(X)$ where $\Sigma$=$\lbrace Y(H):H\in\mathscr{H}\rbrace$. By proposition 2.12, there exist disjoint subfamilies $\Phi_{1}$ and $\Phi_{2}$ such that $T=\Phi_{1}\cup\Phi_{2}$ and card($\Phi_{1}\cap$ $Y(H))$=card($\Phi_{2}\cap$ $Y(H)$)=card($T$) for every $H\in \mathscr{H}$. Consider the unions $\bigcup\limits_{t\in\Phi_{1}}E_{t}$ and $\bigcup\limits_{t\in\Phi_{2}}E_t$. We claim that each of these union is a full subset of $E$. For otherwise, there would exist some $A\in\mathscr{B}(\mathscr{C})\backslash\mathscr{M}(\mathscr{C})$ such that $A\cap S\in\mathscr{B}(\mathscr{C})\backslash\mathscr{M}(\mathscr{C}$ but $A\cap(\bigcup\limits_{t\in\Phi_{1}}E_{t})\in\mathscr{M}(\mathscr{C})$.But then by Theorem 2.9 and 2.10, there exists some set from the family $\mathscr{H}$ which is disjoint with $\bigcup\limits_{t\in\Phi_{1}}E_{t}$ -a contradiction. Thus $\bigcup\limits_{t\in\Phi_{t}}E_{t}$ is a full subset of $E$. Likewise, we can show that $\bigcup\limits_{t\in\Phi_{2}}E_t$ is also a full subset of $E$.

To prove the remaining part, we proceed as follows: let $\Phi_{1}$ and $\Psi_{1}$ be disjoint subfamilies of $T$ such that $T=\Phi_{1}\cup\Psi_{1}$  and each of the unions $\bigcup\limits_{t\in\Phi_{1}}E_{t}$ and  $\bigcup\limits_{t\in\Psi_{1}}E_{t}$ is a full subsets of E. As $S$ is a hull of $E$, it is a hull of $\bigcup\limits_{t\in\Psi_{t}}E_{t}$ as well and using a  procedure similar as above , we find disjoint subfamilies $\Phi_{2}$ and $\Psi_{2}$ of $\Psi_1$ such that $\Psi_1=\Phi_2\cup\Psi_2$ and each of the union $\bigcup\limits_{t\in\Phi_{2}}E_{t}$ and $\bigcup\limits_{t\in\Psi_{2}}E_{t}$  is a full subset of $E$. We may continue with this process indefinitely to get mutually disjoint subfamilies $\Phi_{2},\Phi_{3},$.......,...$\Phi_{n},$......... such that each of the union $\bigcup\limits_{t\in\Phi_{n}}E_{t}$ (n=2,3,....) is a full subset of $E$. Finally we replace $\Phi_{1}$ by $T\backslash\bigcup\limits_{n=2}^{\infty}\Phi_{n}$. This proves the theorem.
\end{proof}
\begin{corollary} \textit{Let $E$ be an abundant set in a Baire base ($X,\mathscr{C}$) and $E=\bigcup\limits_{t\in T}E_{t}$ is a partition of $E$ into meager sets in ($X,\mathscr{C}$). Then there exist into two disjoint subfamilies $\Phi_{1}$ and $\Phi_{2}$ of $T$ such that $T=\Phi_{1}\cup\Phi_{2}$ and the two union $\bigcup\limits_{t\in\Phi_{1}}E_{t}$ and $\bigcup\limits_{t\in\Phi_{2}}E_{t}$
cannot be separated by Baire sets. Consequently, there exists a countably infinite sequence $ \lbrace \Phi_{n}\overunderset{\infty}{n=1}{\rbrace} $ of mutually disjoint subfamilies such that $T=\bigcup\limits_{n=1}^{\infty}\Phi_{n}$ and for no two distinct subfamilies, the corresponding unions can be separated by Baire sets.}
\end{corollary}
Since every set is a partition of singletons, so the following corollary of Theorem 2.11 generalizes Theorem 1.3 of Grzegorek and Labuda.
\begin{corollary}
    \textit{Let E be an abundant set in a point-meager Baire base (X,$\mathscr{C}$). Then there exist disjoint subsets $E_{1}$ and $E_{2}$ of $E$ such that $E=E_{1}\cup E_{2}$ and each of the sets $E_{1}$ and $E_{2}$ is a full subset of $E$. Consequently, there exists a countably infinite sequence $ \lbrace E_{n}\overunderset{\infty}{n=1}{\rbrace} $ of mutually disjoint subsets of $E$ such that $E=\bigcup\limits_{n=1}^{\infty}E_{n}$ and each of the sets $E_{n}$(n=1,2,3,......) is a  full subset of E and therefore no two distinct members in the collection can be separated by Baire sets.}
\end{corollary}
Here we present a variant of Theorem 2.11 where instead of using any combinatorial theorem, we use a Bernstein type of construction which is not dependent on any of the aforestated assumptions.It may be checked (Theorem 33,I,Ch 5, [6]) that the pair ($\mathbb{R},\mathscr{P}$), where $\mathscr{P}$ denotes the family of all perfect sets in $\mathbb{R}$ is a perfect base (Example 2.4) known as the Marczewski base and in this category base, the singular, meager, abundant and Baire sets constitutes the Marczewski classification of sets. The Baire sets in this category base are called Marczewski($s$) sets; the singular, meager$(s_{0})$ and abundant sets are called Marczewski singular, Marczewski meager and Marczewski abundant sets respectively.

\begin{theorem} In the Marczewski base ($\mathbb{R},\mathscr{P}$), if $\mathscr{A}=\lbrace A_{\alpha}:\alpha\in\Lambda \rbrace$ is a family of mutually disjoint $s_{0}$-sets such that $\bigcup\limits_{\alpha\in\Lambda} A_{\alpha}$
is not in $s_0$, then there exist disjoint subfamilies  $\Lambda_{1}$ and $\Lambda_{2}$ of $\Lambda$ such that $\Lambda=\Lambda_{1}\cup\Lambda_{2}$ and each of the union $\bigcup\limits_{\alpha\in\Lambda_{1}}A_{\alpha}$ and $\bigcup\limits_{\alpha\in\Lambda_{2}}A_{\alpha}$ is a full subset of $\bigcup\limits_{\alpha\in \Lambda}A_{\alpha}$. Consequently,there exist a countably infinite sequence $\lbrace \Lambda_{n} \overunderset {\infty}{n=1} \rbrace$of disjoint subfamilies of $\Lambda$ such that $\Lambda=\bigcup\limits_{n=1}^{\infty}\Lambda_{n}$ and each of the union $\bigcup\limits_{\alpha\in\Lambda_{n}} A_{\alpha}$(n=1,2,...) is a full subset of $\bigcup\limits_{\alpha\in\Lambda}A_{\alpha}$ and therefore cannot be separated by s-sets.
\end{theorem}
\begin{proof} Let $B=\bigcup\limits_{\alpha\in\Lambda}A_{\alpha}$. Since $B$ is not $s_{0}$, so it is abundant everywhere in some perfect set $P$ by the Fundamental Theorem (II, Ch 2, [6]). To start with, we show that card($\mathscr{A}_{p})=2$$^{\omega}$(the cardinality of the continuum) where $\mathscr{A}_{p}=\lbrace A_{\alpha}\cap P:A_{\alpha}\in\mathscr{A}\rbrace$.

Divide $\mathscr{A}$ into disjoint subfamilies $\mathscr{A}_{0}$ and $\mathscr{A}_{1}$ and choose perfect sets $P_{0},P_{1}(\subset P)$ such that $B\cap$$P_{o}$ and $B\cap$$P_{1}$ are abundant, $B\cap P_{0}\subset\cup\mathscr{A}_{0}$ and $B\cap P_{1}\subset\cup\mathscr{A}_{1}$.Suppose, at the nth-stage  we have already constructed subfamilies $\lbrace \mathscr{A}_{h}:h\in2^n\rbrace$ such that-
\begin{enumerate}
    \item $\mathscr{A}_{h}\subset\mathscr{A}_{h| n-1}$
    \item $\mathscr{A}_{h}\cap\mathscr{A}_{h'}=\phi$ if $h,h'(\in 2^n)$, $h\neq h'$.
    \item there are disjoint perfect sets $P_{h}\subset P_{h| n-1}$ such that $B\cap$P$_{h}$ is abundant and $B\cap P_{h}\subset\cup\mathscr{A}_{h}$.
\end{enumerate}
Now fix $h\in 2^n$ and divide $\mathscr{A}_{h}$ into disjoint subfamilies $\mathscr{A}_{\hat{h_{0}}}$ and $\mathscr{A}_{\hat{h_{1}}}$ such that for some disjoint perfect sets $P_{\hat{h 0}},P_{\hat{h1}}\subset P_h$ both $B\cap P_{\hat{h0}}$ and $B\cap P_{\hat{h1}}$ are abundant; $B\cap P_{\hat{h0}}\subset\cup\mathscr{A}_{\hat{h0}}$ and $B\cap P_{\hat{h1}}\subset\cup\mathscr{A}_{\hat{h1}}$  If we continue with this construction we obtain a collection $\lbrace\mathscr{A}_{f}:f\in2^\omega\rbrace$ of subfamilies satisfying the following set of  properties:
\begin{enumerate}[label=(\roman*)]
    \item $\mathscr{A}_{f|n}\subset\mathscr{A}_{f|m}$ for any $f\in2^\omega$ and $m<n.$
    \item $\mathscr{A}_{f|n}\subset\mathscr{A}_{f'|n}$ for any $f,f'\in 2^\omega$, $f \neq f'$.
    \item $\cup\lbrace\mathscr{A}_{f|n}:f\in2^\omega\rbrace=\cup\mathscr{A}$
    \item there are disjoint perfect sets $P_{f|n}\subset P_{f|_{n-1}}$ such that $B\cap P_{f|n}$ is abundant and $B\cap P_{f| n}\subset\cup\mathscr{A}_{f|n}$
\end{enumerate}
Let $Q=\lbrace  \bigcap \limits_{n=1}^{\infty}P_{f|n}:f\in2^\omega\rbrace$. Clearly $Q$ is a perfect set contained in $P$ and $card(B\cap Q)$=$2^\omega$, for otherwise, by Theorem 24, I, Ch 5,[6], there would exist a perfect set $T\subset Q$ such that $B\cap T=\phi$, contradicting the fact that $B$ is abundant everywhere in $P$. Also, our process of construction suggests that the cardinality of the set $\lbrace \alpha\in\Lambda:A_{\alpha}\cap Q\neq\phi\rbrace$ is $2^\omega$.

Let $\lbrace P_{\alpha}:\alpha<2^\omega\rbrace$ be an enumeration of all perfect subsets of $P$ and for each $\alpha\in\Lambda$, let $\Lambda_{\alpha}=\lbrace\beta\in\Lambda:A_{\beta}\cap P_{\alpha}\neq\phi\rbrace $. From above, for each $\alpha\in\Lambda$, card($\Lambda_{\alpha})=2^\omega$, so that we choose disjoint sets $A_{\alpha}^{0}$, $A_{\alpha}^{1} \in\lbrace A_{\beta}:\beta\in\Lambda_{\alpha}\rbrace\backslash(\lbrace A_{\gamma}^{0}:\gamma<\alpha\rbrace\cup\lbrace A_{\gamma}^{1}:\gamma<\alpha\rbrace)$ and form disjoint  subfamilies $\lbrace A_{\alpha}^{0}:\alpha\in\Lambda\rbrace$ and $\lbrace A_{\alpha}^{1}:\alpha\in\Lambda\rbrace$.Let $C=\bigcup\lbrace A_{\alpha}^{0}:\alpha\in\Lambda\rbrace $ and $D=\bigcup\lbrace A_{\alpha}^{1}:\alpha\in\Lambda\rbrace$ Then $C$ and $D$ are full subsets of  $B$. As $D$ is contained in the complement of $C$ in $B$, so $C$ and its complement in $B$ are two disjoint union each of which is a full subset of $\bigcup\limits_{\alpha\in\Lambda}A_{\alpha}$. A countable repetition of this argument proves the theorem.
\end{proof}
Marczweski base is a perfect base (Example 2.4) and since every perfect base is a point-meager base (Theorem 3,I,Ch 5,[6]), so follows the corollary stated below.

\begin{corollary}  In the Marczewski base ($\mathbb{R},\mathscr{P})$, if $E$ is a set not in $s_{0}$, then there exist disjoint subsets $E_1$ and $E_2$ of $E$ such that $E=E_1\cup E_2$ and each of the sets $E_1$ and $E_2$  is a full subset of $E$. Consequently, there exists a countably infinite sequence $\lbrace E_n\overunderset{\infty}{n=1}\rbrace$ of mutually $s_0$ disjoint subsets of $E$ such that $E=\bigcup\limits_{n=1}^{\infty}E_n$ and each of the sets $E_n$ (n=1,2,...) is a full subset of $E$ and so no two distinct members in the collection can be separated by $s$-sets.
\end{corollary}

 \bibliographystyle{plain}

\end{document}